\documentclass[12pt]{article}
\usepackage{graphicx}
\usepackage{amsmath}
\usepackage{amssymb}
\usepackage{theorem}

\usepackage{fancyhdr}

\numberwithin{equation}{section}

\newtheorem{thm}{Theorem}[section]
\newtheorem{lemma}[thm]{Lemma}

\newtheorem{prop}[thm]{Proposition}
\newtheorem{cor}[thm]{Corollary}
{\theorembodyfont{\rmfamily}
\newtheorem{defn}[thm]{Definition}

\newtheorem{rmk}[thm]{Remark}
}

\newcommand{\qed}{\hfill \mbox{\raggedright \rule{.07in}{.1in}}}
 
\newenvironment{proof}{\vspace{1ex}\noindent{\bf
Proof}\hspace{0.5em}}{\hfill\qed\vspace{1ex}}
\newenvironment{pfof}[1]{\vspace{1ex}\noindent{\bf Proof of
#1}\hspace{0.5em}}{\hfill\qed\vspace{1ex}}

\newcommand{\R}{{\mathbb R}}

\newcommand{\C}{{\mathbb C}}

\newcommand{\N}{{\mathbb N}}

\newcommand{\cF}{{\mathcal F}}

\newcommand{\sumj}{{\textstyle\sum_j}}

\newcommand{\tY}{{\widetilde Y}}
\newcommand{\tR}{\widetilde R}
\newcommand{\tF}{\widetilde F}
\newcommand{\tmu}{\tilde\mu}
\newcommand{\tvarphi}{\widetilde\varphi}
\newcommand{\hR}{\widehat R}
\newcommand{\hT}{\widehat T}
\newcommand{\hU}{\widehat U}
\newcommand{\hV}{\widehat V}
\newcommand{\hA}{\widehat A}
\newcommand{\hB}{\widehat B}

\newcommand{\hE}{\widehat E}
\newcommand{\hG}{\widehat G}
\newcommand{\hL}{\widehat L}
\newcommand{\hH}{\widehat H}

\newcommand{\eps}{\epsilon}
\newcommand{\supp}{\operatorname{supp}}

\newcommand{\Leb}{\operatorname{Leb}}
\newcommand{\spec}{\operatorname{spec}}

\newcommand{\SMALL}{\textstyle}
\newcommand{\BIG}{\displaystyle}

\title{Decay in norm of transfer operators for semiflows}

\author{Ian Melbourne\thanks{Mathematics Institute, University of Warwick, Coventry, CV4 7AL, UK}
\and Nicol\`o Paviato
\thanks{Mathematics Institute, University of Warwick, Coventry, CV4 7AL, UK}
 \and Dalia Terhesiu\thanks{Mathematisch Instituut,
University of Leiden, Niels Bohrweg 1, 2333 CA Leiden, Netherlands}}

\date{9 April 2021}

\begin{document}

 \maketitle

\begin{abstract}
We establish exponential decay in H\"older norm of transfer operators applied to smooth observables of uniformly and nonuniformly expanding semiflows with exponential decay of correlations.  
\end{abstract}

 \section{Introduction} 
 \label{sec-intro}

Exponential decay of correlations is well-understood for large classes of uniformly and nonuniformly expanding maps, see for example~\cite{Bowen75,GoraBoyarsky89,HofbauerKeller82,Keller85,Liverani13,Ruelle78,Rychlik83,Saussol00,Sinai72,Young98}.
The typical method of proof is to establish a spectral gap for the associated transfer operator $L$.  Such a spectral gap yields a decay rate $\|L^n v-\int v\|\le C_v e^{-an}$ for $v$ lying in a suitable function space, where $a$, $C_v$ are positive constants.  Decay of correlations is an immediate consequence of such decay for $L^n$.

Results on decay of correlations lead to
numerous statistical limit theorems. Although not needed for results such as the central limit theorem, strong norm control on $L^nv$ is often useful for finer statistical properties.  For example, rates of convergence in the central limit theorem~\cite{Gouezel05} and the associated functional central limit theorem~\cite{AntoniouM19} rely heavily on control of operator norms.

In this paper, we consider norm decay of transfer operators for uniformly and nonuniformly expanding semiflows.  Here, the standard method is to deduce decay of the correlation function from analyticity of Laplace transforms, bypassing spectral properties of $L_t$, see~\cite{Dolgopyat98a,Liverani04,Pollicott85}. As far as we know, the only result on spectral gaps for transfer operators of semiflows is due to Tsujii~\cite{Tsujii08}.  However, this result is for suspension semiflows over the doubling map with a $C^3$ roof function, where the smoothness of the roof function is crucial and very restrictive.  A similar result for contact Anosov flows is proved in~\cite{Tsujii10}.
Both of the papers~\cite{Tsujii08,Tsujii10} obtain spectral gaps for $L_t$ acting on a suitable anisotropic Banach space. 
Apart from these, there are apparently no previous results on norm decay of transfer operators for semiflows and flows.

Recently, in~\cite{MPTapp}, we showed that spectral gaps are impossible in H\"older spaces with exponent greater than $\frac12$ (and in any Banach space that embeds in such a H\"older space).  
Nevertheless, our aim of controlling the H\"older norm of $L_tv$ for a large class of semiflows and observables $v$ remains viable, 
and our main result is the first in this direction.
We consider uniformly and nonuniformly expanding semiflows satisfying a Dolgopyat-type estimate~\cite{Dolgopyat98a}.  Such an estimate plays a key role in proving exponential decay of correlations for the semiflow.
Theorem~\ref{thm:normdecay} below shows how to use this estimate to prove exponential decay of
$L_tv$ in a H\"older norm for smooth mean zero observables satisfying a good support condition.   
Apart from the Dolgopyat estimate, the main ingredient is an operator renewal equation for semiflows~\cite{MT17} which enables consideration of the operator Laplace transform $\int_0^\infty e^{-st}L_t\,dt$.

The remainder of the paper is organised as follows.
In Section~\ref{sec:normdecay}, we recall the setup for nonuniformly expanding semiflows with exponential decay of correlations and state our main result, Theorem~\ref{thm:normdecay}, on decay in norm.
In Section~\ref{sec:proof}, we prove Theorem~\ref{thm:normdecay}.

\vspace{-2ex}
\paragraph{Notation}
We use ``big O'' and $\ll$ notation interchangeably, writing $a_n=O(b_n)$ or $a_n\ll b_n$
if there are constants $C>0$, $n_0\ge1$ such that
$a_n\le Cb_n$ for all $n\ge n_0$.

\section{Setup and statement of the main result}
\label{sec:normdecay}

In this section, we state our result on H\"older norm decay of transfer operators for uniformly and nonuniformly expanding semiflows.

Let $(Y,d)$ be a bounded metric space with Borel probability measure $\mu$ and an at most countable measurable partition $\{Y_j\}$.
Let $F:Y\to Y$  be a measure-preserving transformation such that $F$ restricts to a measure-theoretic bijection from $Y_j$ onto $Y$ for each $j$.
Let $g=d\mu/(d\mu\circ F)$ be the inverse Jacobian of $F$.  

Fix $\eta\in(0,1)$.  
Assume that there are constants $\lambda>1$ and $C>0$ such that
$d(Fy,Fy')\ge \lambda d(y,y')$ and 
$|\log g(y)-\log g(y')|\le Cd(Fy,Fy')^\eta$
for all $y,y'\in Y_j$, $j\ge1$.
In particular, $F$ is a Gibbs-Markov map as in~\cite{AD01} (see also~\cite{Aaronson,ADU93}) with ergodic (and mixing) invariant measure $\mu$.  

Let
$\varphi:Y\to[2,\infty)$ be a piecewise continuous roof function.
We assume that there is a constant $C>0$ such that
\begin{equation} \label{eq:phi}
|\varphi(y)-\varphi(y')|\le Cd(Fy,Fy')^\eta
\end{equation}
for all $y,y'\in Y_j$, $j\ge1$.
Also, we assume exponential tails, namely that there exists $\delta_0>0$ such that
\begin{equation} \label{eq:exp}
\sumj \mu(Y_j)e^{\delta_0|1_{Y_j}\varphi|_\infty} 
<\infty.
\end{equation}

Define the suspension 
$Y^\varphi=\{(y,u)\in Y\times[0,\infty):u\in [0,\varphi(y)]\}/\sim$
where $(y,\varphi(y))\sim (Fy,0)$.
The suspension semiflow $F_t:Y^\varphi\to Y^\varphi$ is given by
$F_t(y,u)=(y,u+t)$ computed modulo identifications.
We define the ergodic $F_t$-invariant probability measure
$\mu^\varphi=(\mu\times{\rm Lebesgue})/\bar\varphi$
where $\bar\varphi=\int_Y\varphi\,d\mu$.
\footnote{We call such semiflows ``nonuniformly expanding'' since they are the continuous time analogue of maps that are nonuniformly expanding in the sense of Young~\cite{Young98}. ``Uniformly expanding'' semiflows are those with $\varphi$ bounded; they have bounded distortion as well as uniform expansion.}

Let $L_t:L^1(Y^\varphi)\to L^1(Y^\varphi)$ denote the transfer operator corresponding to $F_t$ (so $\int_{Y^\varphi} L_tv\,w\,d\mu^\varphi=\int_{Y^\varphi} v\,w\circ F_t\,d\mu^\varphi$ for all $v\in L^1(Y^\varphi)$, $w\in L^\infty(Y^\varphi)$, $t>0$)
and let $R_0:L^1(Y)\to L^1(Y)$ denote the transfer operator for $F$.
Recall (see for example~\cite{AD01}) that 
$(R_0v)(y)=\sumj g(y_j)v(y_j)$ where $y_j$ is the unique preimage of $y$ under $F|Y_j$, and there is a constant $C>0$ such that
\begin{equation} \label{eq:GM}
|g(y)|\le C\mu(Y_j), \qquad
|g(y)-g(y')|\le C\mu(Y_j)d(Fy,Fy')^\eta,
\end{equation} 
for all $y,y'\in Y_j$, $j\ge1$.

\paragraph{Function space on $Y^\varphi$}
Let $Y^\varphi_j=\{(y,u)\in Y^\varphi:y\in Y_j\}$.
Fix $\eta\in(0,1]$, $\delta>0$.  
For $v:Y^\varphi\to\R$, define
$|v|_{\delta,\infty}=\sup_{(y,u)\in Y^\varphi} e^{-\delta u}|v(y,u)|$
and
\[
\|v\|_{\delta,\eta}=|v|_{\delta,\infty}+|v|_{\delta,\eta},\qquad 
|v|_{\delta,\eta}=\sup_{j\ge1}\sup_{(y,u),(y',u)\in Y_j^\varphi,\,y\neq y'} e^{-\delta u}\frac{|v(y,u)-v(y',u)|}{d(y,y')^\eta}.
\]
Then $\cF_{\delta,\eta}(Y^\varphi)$ consists of observables $v:Y^\varphi\to\R$ with
$\|v\|_{\delta,\eta}<\infty$.

Next, define $\partial_uv$ to be the partial derivative of $v$ with respect to $u$ at points $(y,u)\in Y^\varphi$ with $u\in(0,\varphi(y))$ and to be the appropriate one-sided partial derivative when $u\in\{0,\,\varphi(y)\}$.
For $m\ge0$, define $\cF_{\delta,\eta,m}(Y^{\varphi})$ to consist of observables
$v:Y^\varphi\to\R$ such that $\partial_u^jv\in \cF_{\delta,\eta}(Y^\varphi)$ 
for $j=0,1,\dots,m$, with norm
$\|v\|_{\delta,\eta,m}=\max_{j=0,\dots,m} \|\partial_u^jv\|_{\delta,\eta}$.

\begin{defn} \label{def:good}
We say that a function $u:Y^\varphi\to\R$ has {\em good support} if there exists
$r>0$ such that
$\supp v\subset\{(y,u)\in Y\times\R:u\in[r,\varphi(y)-r]\}$.
\end{defn}

For functions with good support, $\partial_uv$ coincides with the derivative $\partial_tv=\lim_{h\to0} (v\circ F_h-v)/h$ in the flow direction.

Let
\[
\cF_{\delta,\eta,m}^0(Y^{\varphi})=\{v\in \cF_{\delta,\eta,m}(Y^{\varphi}):{\SMALL\int}_{Y^\varphi}v\,d\mu^\varphi=0\}.
\]
We write $\cF_{\delta,\eta}(Y^\varphi)$ 
and $\cF_{\delta,\eta}^0(Y^\varphi)$ when $m=0$.

\paragraph{Function space on $Y$}
For $v:Y\to\R$, define
\[
\|v\|_\eta=|v|_\infty+|v|_\eta,
\qquad
|v|_\eta=\sup_{j\ge1}\sup_{y,y'\in Y_j,\,y\neq y'}|v(y)-v(y')|/d(y,y')^\eta.
\]
Let $\cF_\eta(Y)$ consist of observables $v:Y\to\R$ with
$\|v\|_\eta<\infty$.

\paragraph{Dolgopyat estimate}

Define the twisted transfer operators
\[
\hR_0(s):L^1(Y)\to L^1(Y), \qquad
\hR_0(s)v=R_0(e^{-s\varphi}v).
\]
We assume that there exists
$\gamma\in(0,1)$, $\eps>0$, $m_0\ge0$, $A,D>0$ such that
\begin{align} \label{eq:Dolg}
{\|\hR_0(s)^n\|}_{\cF_\eta(Y)\mapsto \cF_\eta(Y)} \le |b|^{m_0}\gamma^n
\end{align}
 for all $s=a+ib\in\C$ with $|a|<\eps$, $|b|\ge D$ and all $n\ge A\log|b|$.  Such an assumption holds in the settings of~\cite{AraujoM16,AGY06, BaladiVallee05,Dolgopyat98a}.  

Now we can state our main result on norm decay for $L_t$.

\begin{thm} \label{thm:normdecay}
Under these assumptions, there exists $\eps>0$, $m\ge1$, $C>0$ such that
\[
\|L_tv\|_{\delta,\eta,1}
\le Ce^{-\eps t}\|v\|_{\delta,\eta,m}
\quad\text{for all $t>0$}  
\]
for all $v\in \cF_{\delta,\eta,m}^0(Y^{\varphi})$ with good support.
\end{thm}

\begin{rmk} 
Since the norm applied to $v$ is stronger than the norm applied to $L_tv$, Theorem~\ref{thm:normdecay} does not imply a spectral gap for $L_t$.
We note that the norm on $\cF_{\delta,\eta,1}$ gives no H\"older control in the flow direction when passing through points of the form $(y,\varphi(y))$.
This lack of control is a barrier to mollification arguments of the type usually used to pass from smooth observables to H\"older observables.
In fact, such arguments are doomed to fail at the operator level by~\cite[Theorem~1.1]{MPTapp} when $\eta>\frac12$ and hence seem unlikely for any $\eta$.
\end{rmk}

\begin{rmk} 
Usually, we can take $m_0\in(0,1)$ in~\eqref{eq:Dolg} in which case $m=3$ suffices in Theorem~\ref{thm:normdecay}.

There are numerous simplifications when $\{Y_j\}$ is a finite partition. 
In particular, conditions~\eqref{eq:phi} and~\eqref{eq:exp} are redundant and we can take $\delta=0$.  
\end{rmk}

\section{Proof of Theorem~\ref{thm:normdecay}}
\label{sec:proof}

Our proof of norm decay is broken into three parts.
In Subsection~\ref{sec:renewal}, we recall a continuous-time operator renewal equation~\cite{MT17} which enables estimates of Laplace transforms of transfer operators at the level of $Y$.  In Subsection~\ref{sec:TL}, we show how to pass to estimates of Laplace transforms of $L_t$.
In Subsection~\ref{sec:contour}, we invert the Laplace transform to obtain norm decay of $L_t$.

\subsection{Operator renewal equation}
\label{sec:renewal}

Let $\tY=Y\times[0,1]$ and define 
\[
\tF:\tY\to\tY, \qquad \tF(y,u)=(Fy,u),
\]
with transfer operator $\tR:L^1(\tY)\to L^1(\tY)$. 
Also, define 
\[
\tvarphi:\tY\to[2,\infty), \quad \tvarphi(y,u)=\varphi(y).
\]
Define the twisted transfer operators
\[
\hR(s):L^1(\tY)\to L^1(\tY), \qquad
\hR(s)v=\tR(e^{-s\tvarphi}v).
\]

Let $\tY_j=Y_j\times[0,1]$.
For $v:\tY\to\R$, define
\[
\|v\|_\eta=|v|_\infty+|v|_\eta,
\qquad
|v|_\eta=\sup_{j\ge1}\sup_{(y,u),(y',u)\in\tY_j,\,y\neq y'}|v(y,u)-v(y',u)|/d(y,y')^\eta.
\]
Let $\cF_\eta(\tY)$ consist of observables $v:\tY\to\R$ with
$\|v\|_\eta<\infty$.
Let
\[
\cF_\eta^0(\tY)=\{v\in \cF_\eta(\tY):{\SMALL\int}_\tY v\,d\tmu=0\}
\]
where $\tmu=\mu\times\Leb_{[0,1]}$.

\begin{lemma} \label{lem:main}
Write $s=a+ib\in\C$.
There exists $\eps>0$, $m_1\ge0$, $C>0$ such that
\begin{itemize}
\item[(a)]
$s\mapsto(I-\hR(s))^{-1}:\cF_\eta^0(\tY)\to \cF_\eta(\tY)$ is analytic on
$\{|a|< \eps\}$;
\item[(b)]
$s\mapsto(I-\hR(s))^{-1}:\cF_\eta(\tY)\to \cF_\eta(\tY)$ is analytic on
$\{|a|< \eps\}$ except for a simple pole at $s=0$;
\item[(c)]
\(
{\|(I-\hR(s))^{-1}\|}_{\cF_\eta(\tY)\mapsto \cF_\eta(\tY)}\le C|b|^{m_1}
\)
for $|a| \le\eps$, $|b|\ge1$.
\end{itemize}
\end{lemma}

\begin{proof}
It suffices to verify these properties  for
$Z(s)=(I-\hR_0(s))^{-1}$ on $Y$. They immediately transfer to
$(I-\hR(s))^{-1}$ on $\tY$ since
$(\hR v)(y,u)=(\hR_0 v^u)(y)$
where $v^u(y)=v(y,u)$.

The arguments for passing from~\eqref{eq:Dolg} to the desired properties for $Z(s)$ are standard.
For completeness, we sketch these details now recalling arguments from~\cite{AraujoM16}.
Define $\cF_\eta(Y)$ with norm $\|\;\|_\eta$ by restricting to $u=0$
(this coincides with the usual H\"older space on $Y$).
Let $A$, $D$, $\eps$ and $m_0$ be as in~\eqref{eq:Dolg}.
Increase $A$ and $D$ so that $D>1$ and
$|b|^{m_0}\gamma^{[A\log|b|]}\le\frac12$ for $|b|\ge D$.
Suppose that $|a|\le\eps$, $|b|\ge D$.
Then $\|\hR_0(s)^{[A\log|b|]}\|_\eta\le |b|^{m_0}\gamma^{[A\log|b|]}\le\frac12$ and
\(
\|(I-\hR_0(s)^{[A\log|b|]})^{-1}\|_\eta\le 2.
\)

As in~\cite[Proposition~2.5]{AraujoM16}, we can shrink $\eps$ so
that $s\to \hR_0(s)$ is continuous on $\cF_\eta(Y)$ for $|a|\le\eps$.
The simple eigenvalue $1$ for $\hR_0(0)=R_0$ extends to a continuous family of simple eigenvalues $\lambda(s)$ for $|s|\le \eps$.
Hence we can choose $\eps$ so that
$\frac12<\lambda(a)<2$ for $|a|\le \eps$.
By~\cite[Corollary~2.8]{AraujoM16}, $\|\hR_0(s)^n\|_\eta\ll |b|\lambda(a)^n\le |b|2^n$ for all $n\ge1$, 
$|a|\le \eps$, $|b|\ge D$.  Hence 
\begin{align*}
\|Z(s)\|_\eta
& \le \big(1+\|\hR_0(s)\|_\eta +\dots+\|\hR_0(s)^{[A\log|b|]-1}\|_\eta\big)\|(I-\hR_0(s)^{[A\log|b|]})^{-1}\|_\eta
\\ & \ll  (\log|b|)\, |b|\,2^{A\log|b|}
\le |b|^{m_1},
\end{align*}
with $m_1=1+A\log 2$.
This proves analyticity on the region $\{|a|<\eps, |b|>D\}$ with the desired estimates for property~(c) on this region.

For $|a|\le\eps$, $|b|\le D$, we recall arguments from the proof of~\cite[Lemma~2.22]{AraujoM16} (where $\hR_0(s)$ is denoted $Q_s$).
For $\eps$ sufficiently small, 
the part of spectrum of $\hR_0(s)$ that is close to $1$ consists only of isolated eigenvalues.
Also, the spectral radius of $\hR_0(s)$ is at most $\lambda(a)$ and $\lambda(a)<1$ for
$a\in[0,\eps]$, so $s\mapsto Z(s)$ is analytic on $\{0<a<\eps\}$.

Suppose that $\hR_0(ib)v=v$ for some $v\in\cF_\eta(Y)$, $b\neq0$.  
Choose $q\ge1$ such that $q|b|>D$.  Since
$\hR_0(s)$ is the $L^2$ adjoint of $v\mapsto e^{s\varphi}v\circ F$, we have
$e^{ib\varphi}v\circ F=v$.  Hence $e^{iqb\varphi}v^q\circ F=v^q$ and so
$\hR_0(iqb)v^q=v^q$.  But
$\|Z(iqb)v^q\|_\eta<\infty$, so $v=0$.  Hence $1\not\in \spec\hR_0(ib)$ for all $b\neq0$.
It follows that for all $b\neq0$ there exists an open set $U_b\subset\C$ containing $ib$ such that $1\not\in\spec\hR_0(s)$ for all $s\in U_b$, and so
$s\mapsto Z(s)$ is analytic on $U_b$.

Next, we recall that for $s$ near to zero, 
$\lambda(s)=1+cs+O(s^2)$ where $c<0$.
Hence $s\mapsto Z(s)$ has a simple pole at zero.
It follows that there exists $\eps>0$ such that $s\mapsto Z(s)$ is analytic
on $\{|a|<\eps,|b|<2D\}$ except for a simple pole at $s=0$.  Combining this with the estimates on $\{|a|<\eps,|b|\ge D\}$ we have proved properties~(b) and (c) for $Z(s)$. 

Finally, the spectral projection $\pi$ corresponding to the eigenvalue $\lambda(0)=1$ for $\hR_0(0)=R$ is given by
$\pi v=\int_Y v\,d\mu$.  Hence the pole disappears on restriction to observables of mean zero, proving property~(a) for $Z(s)$.
\end{proof}

Next define
\[
T_tv=1_\tY L_t(1_\tY v), \qquad
U_tv=1_\tY L_t(1_{\{\tvarphi>t\}} v)
\]
and
\[
\qquad
\hT(s)=\int_0^\infty e^{-st}T_t\,dt,
\qquad
\hU(s)=\int_0^\infty e^{-st}U_t\,dt,
\]

By~\cite[Theorem~3.3]{MT17}, we have the operator renewal equation
\[
\hT=\hU(I-\hR)^{-1}.
\]

\begin{prop} \label{prop:U}
 There exists $\eps>0$, $C>0$ such that
$s\mapsto \hU(s):\cF_\eta(\tY)\to\cF_\eta(\tY)$ is analytic on $\{|a|<\eps\}$ and
${\|\hU(s)\|}_{\cF_\eta(\tY)\mapsto \cF_\eta(\tY)}\le C|s|$ for $|a|\le\eps$.
\end{prop}

\begin{proof}
By~\cite[Proposition~3.4]{MT17},
\[
(U_tv)(y,u)=\begin{cases} v(y,u-t)1_{[t,1]}(u) & 0\le t\le 1
\\
(\tR v_t)(y,u) & t>1
\end{cases}
\]
where $v_t(y,u)=1_{\{t<\varphi(y)<t+1-u\}}v(y,u-t+\varphi(y))$.
Hence $\hU(s)=\hU_1(s)+\hU_2(s)$ where
\[
(\hU_1(s)v)(y,u)=\int_0^u e^{-st}v(y,u-t)\,dt, \qquad
\hU_2(s)v = \int_1^\infty e^{-st} \tR v_t\,dt.
\]
It is clear that $\|\hU_1(s)v\|_\eta\le e^\eps \|v\|_\eta$.
We focus attention on the second term 
\[
(\hU_2(s)v)(y,u)=\sumj g(y_j)\int_1^\infty e^{-st} v_t(y_j,u)\,dt
=\sumj g(y_j)\hV(s)(y_j,u),
\]
where $\hV(s)(y,u)=\int_u^1 e^{s(t-u-\varphi)} v(y,t)\,dt$.
Clearly, $|1_{Y_j}\hV(s)|_\infty 
\le e^{\eps|1_{Y_j}\varphi|_\infty}|v|_\infty$.
Also, 
\[
\hV(s)(y,u)-\hV(s)(y',u) = I+J,
\]
where
\begin{align*}
I & =\int_u^1 (e^{s(t-u-\varphi(y))}-e^{s(t-u-\varphi(y'))})v(y,t)\,dt, \\
J & =\int_u^1 e^{s(t-u-\varphi(y'))}(v(y,t)-v(y',t))\,dt.
\end{align*}
For $y,y'\in Y_j$,
\[
|I|\le |v|_\infty\int_u^1 e^{\eps(|1_{Y_j}\varphi|_\infty+u-t)}|s||\varphi(y)-\varphi(y')|\,dt
\ll |s||v|_\infty\, e^{\eps|1_{Y_j}\varphi|_\infty} d(Fy,Fy')^\eta
\]
by~\eqref{eq:phi},
and
\[
|J| \le
\int_u^1 e^{\eps(|1_{Y_j}\varphi|_\infty+u-t)}|v(y,t)-v(y',t)|\,dt
\le e^{\eps|1_{Y_j}\varphi|_\infty}|v|_\eta\, d(y,y')^\eta.
\]
Hence $|\hV(s)(y,u)-\hV(s)(y',u)|_\eta\ll 
 |s| e^{\eps|1_{Y_j}\varphi|_\infty} \|v\|_\eta\,d(Fy,Fy')^\eta$.

It follows from the estimates for $1_{Y_j}\hV(s)$ together with~\eqref{eq:GM} that
$\|\hU_2(s)v\|_\eta\ll \sumj |s| \mu(Y_j)e^{\eps|1_{Y_j}\varphi|_\infty} \|v\|_\eta$.
By~\eqref{eq:exp},
$\|\hU_2(s)v\|_\eta\ll  |s| \|v\|_\eta$
 for $\eps$ sufficiently small.
We conclude that $\|\hU(s)v\|_\eta\ll  |s| \|v\|_\eta$.
\end{proof}

\subsection{From $\hT$ on $\tY$ to $\hL$ on $Y^\varphi$}
\label{sec:TL}

Lemma~\ref{lem:main} and Proposition~\ref{prop:U} yield analyticity and estimates for $\hT=\hU(I-\hR)^{-1}$ on $\tY$.
In this subsection, we show how these properties are inherited by
$\hL(s)=\int_0^\infty e^{-st}L_t\,dt$ on $Y^\varphi$.

\begin{rmk} The approach in this subsection is similar to that in~\cite[Section~5]{BMT19} but there are some important differences.
The rationale behind the two step decomposition in Propositions~\ref{prop:L=AT} and~\ref{prop:L} below is that
the discreteness of the decomposition in Proposition~\ref{prop:L=AT} simplifies many formulas significantly.
In particular, the previously problematic term $E_t$ in~\cite{BMT19}  becomes elementary (and vanishes for large $t$ when $\varphi$ is bounded).
The decomposition in Proposition~\ref{prop:L} remains continuous to simplify the estimates in Proposition~\ref{prop:BG}.

Since the setting in~\cite{BMT19} is different (infinite ergodic theory, reinducing) we keep the exposition here self-contained even where the estimates coincide with those in~\cite{BMT19}.
\end{rmk}

Define
\begin{alignat*}{2}
 & A_n:L^1(\tY)\to L^1(Y^\varphi),
& \qquad& 
(A_nv)(y,u)=1_{\{n\le u< n+1\}}(L_nv)(y,u),\;n\ge0, \\[.75ex]
& E_t:L^1(Y^\varphi)\to L^1(Y^\varphi), & \qquad &
(E_tv)(y,u)=1_{\{[t]+1\le u\le\varphi(y)\}}(L_tv)(y,u),\;t>0.
\end{alignat*}

\begin{prop} \label{prop:L=AT}
$\BIG L_t=\sum_{j=0}^{[t]}
A_j1_\tY L_{t-j}+
 E_t$ for $t>0$.
\end{prop}

\begin{proof}
For $y\in Y$, $u\in(0,\varphi(y))$,
\begin{align*}
(L_tv)(y,u) & =\sum_{j=0}^{[t]}1_{\{j\le u< j+1\}}(L_tv)(y,u)+
1_{\{[t]+1\le u\le\varphi(y)\}}(L_tv)(y,u)
\\
& =\sum_{j=0}^{[t]}(A_jL_{t-j}v)(y,u)+(E_tv)(y,u).
\end{align*}
Now use that $A_n=A_n1_\tY$.
\end{proof}

Next, define 
\begin{alignat*}{2}
& B_t:L^1(Y^\varphi)\to L^1(\tY), & \qquad &
B_tv=1_\tY L_t(1_{\Delta_t}v),
\\
& G_t:L^1(Y^\varphi)\to L^1(\tY), & \qquad &
G_tv =B_t(\omega(t)v),
\\
& H_t:L^1(Y^\varphi)\to L^1(\tY), 
& \qquad &
H_tv=1_\tY L_t(1_{\Delta'_t}v),
\end{alignat*}
for $t>0$, where
 \begin{align*}
 \Delta_t  & =\{(y,u)\in Y^\varphi: \varphi(y)-t\le u< \varphi(y)-t+1\}
 \\
\Delta'_t  &=\{(y,u)\in Y^\varphi: u< \varphi(y)-t\},
  \qquad  \omega(t)(y,u)  =\varphi(y)-u-t+1 .
 \end{align*}

\begin{prop} \label{prop:L}
$\BIG 1_\tY L_t=\int_0^t T_{t-\tau}B_\tau \,d\tau
+G_t+H_t$ for $t>0$.
\end{prop}

\begin{proof}
Let $y\in Y$, $u\in[0,\varphi(y)]$.  Then
\begin{align*}
\int_0^t 1_{\Delta_\tau}(y,u)\,d\tau
& = \int_0^t  1_{\{\varphi(y)-u\le \tau\le \varphi(y)-u+1\}}\,d\tau
\\ & =1_{\{t\ge \varphi(y)-u+1\}}+
1_{\{\varphi(y)-u\le t<\varphi(y)-u+1\}}(t-\varphi(y)+u)
\\ & =1-1_{\{t< \varphi(y)-u+1\}}+
1_{\{\varphi(y)-u\le t<\varphi(y)-u+1\}}(t-\varphi(y)+u)
\\ & =1-1_{\Delta_t'}(y,u)
+1_{\Delta_t}(y,u)(t-\varphi(y)+u-1).
\end{align*}
Hence
\(
\int_0^t 1_{\Delta_\tau}\,d\tau=1
-1_{\Delta_t}\omega(t) -1_{\Delta'_t}.
\)
It follows that
\begin{align*}
\int_0^t T_{t-\tau}B_\tau & v \,d\tau
 = 1_\tY\int_0^t L_{t-\tau}1_\tY B_\tau v \,d\tau
 = 1_\tY\int_0^t L_{t-\tau}B_\tau v \,d\tau
\\ & = 1_\tY\int_0^t L_{t-\tau}L_\tau (1_{\Delta_\tau} v) \,d\tau
= 1_\tY L_t\Big(\int_0^t 1_{\Delta_\tau} v \,d\tau\Big)
=1_\tY L_tv-G_tv-H_tv
\end{align*}

as required.
\end{proof}

We have already defined the Laplace transforms $\hL(s)$ and $\hT(s)$ for $s=a+ib$ with $a>0$.  Similarly, define
\begin{alignat*}{2}
\hB(s) & =\int_0^\infty e^{-st}B_t\,dt, \qquad &
\hE(s) & =\int_0^\infty e^{-st}E_t\,dt, 
\\
\hG(s) & =\int_0^\infty e^{-st}G_t\,dt, \qquad &
\hH(s) & =\int_0^\infty e^{-st}H_t\,dt.
\end{alignat*}
 Also, we define the discrete transform
\(
\BIG \hA(s)=\sum_{n=0}^\infty e^{-sn}A_n.
\)

\begin{cor} \label{cor:L}
$\hL(s)=\hA(s)\hT(s)\hB(s)
+\hA(s)\hG(s)
+\hA(s)\hH(s)
+\hE(s)$ for $a>0$.
\end{cor}

\begin{proof}
By Proposition~\ref{prop:L=AT},
\begin{align*}
\hL(s)-\hE(s) & =\int_0^\infty e^{-st} \sum_{j=0}^{[t]}
A_j1_\tY L_{t-j}\,dt
= \sum_{j=0}^\infty e^{-sj}A_j1_\tY  \int_{j}^\infty e^{-s(t-j)}L_{t-j}\,dt
\\ & = \hA(s) 1_\tY \int_{0}^\infty e^{-st}  L_{t}\,dt
 =\hA(s)1_\tY \hL(s).
\end{align*}
Hence $\hL=\hA1_\tY \hL+\hE$.
In addition, by Proposition~\ref{prop:L},
$\BIG 1_\tY \hL=\hT\hB +\hG+\hH$.
\end{proof}

\begin{prop} \label{prop:AEH}
Let $\delta>\eps>0$.  
Then there is a constant $C>0$ such that 
\begin{itemize}
\item[(a)]
${\|\hA(s)\|}_{\cF_\eta(\tY)\to\cF_{\delta,\eta}(Y^\varphi)}\le 1$,
\item[(b)]
${\|\hE(s)\|}_{\cF_{\delta,\eta}(Y^\varphi)\to\cF_{\delta,\eta}(Y^\varphi)}\le C$,
\item[(c)]
${\|\hH(s)\|}_{\cF_{\delta,\eta}(Y^\varphi)\to\cF_\eta(\tY)}\le e^\delta$,
\end{itemize}
for $|a| \le \eps$.
\end{prop}

\begin{proof}
(a)  Let $v\in \cF_\eta(\tY)$.
Let $(y,u),\,(y',u)\in Y_j^\varphi$, $j\ge1$.
Since
$(A_nv)(y,u)=1_{\{n\le u< n+1\}}v(y,u-n)$,
\[
(\hA(s)v)(y,u) =\sum_{n=0}^\infty e^{-sn} 1_{\{n\le u< n+1\}}v(y,u-n)
=e^{-s[u] } v(y,u-[u]).
\]
Hence
\[
|(\hA(s)v)(y,u)|\le e^{\eps u}|v|_{\infty},  \quad
|(\hA(s)v)(y,u)-(\hA(s)v)(y',u)|
\le e^{\eps u}|v|_{\eta}\,d(y,y')^\eta .
\]
That is, $|\hA(s)v|_{\eps,\infty}\le |v|_\infty$,
$|\hA(s)v|_{\eps,\eta} \le |v|_\eta$.
Hence $\|\hA(s)v\|_{\delta,\eta}\le
\|\hA(s)v\|_{\eps,\eta}\le\|v\|_\eta$.
\\[.75ex]
(b) We take $C=1/(\delta-\eps)$.
Let $v\in \cF_{\delta,\eta}(Y^\varphi)$.
Let $(y,u),\,(y',u)\in Y_j^\varphi$, $j\ge1$.
Note that
$(E_tv)(y,u)=1_{\{[t]+1\le u\}}v(y,u-t)$, so
\[
(\hE(s)v)(y,u) =\int_0^\infty e^{-st} 1_{\{[t]+1\le u\}}v(y,u-t)\,dt.
\]
Hence
\[
|(\hE(s)v)(y,u)|\le \int_0^\infty e^{\eps t}|v|_{\delta,\infty} \,e^{\delta(u-t)}\,dt
= C |v|_{\delta,\infty} \,e^{\delta u},
\]
and
\[
|(\hE(s)v)(y,u)-(\hE(s)v)(y',u)|
\le \int_0^\infty e^{\eps t}|v|_{\delta,\eta}\,d(y,y')^\eta e^{\delta(u-t)}\,dt
= C e^{\delta u}|v|_{\delta,\eta} \,d(y,y')^\eta.
\]
That is, $|\hE(s)v|_{\delta,\infty}\le |v|_{\delta,\infty}$ and
$|\hE(s)v|_{\delta,\eta} \le |v|_{\delta,\eta}$.
\\[.75ex]
(c)
Let $v\in \cF_{\eps,\eta}(Y^\varphi)$.
Let $(y,u),\,(y',u)\in\tY_j$, $j\ge1$.  Then
$(H_t v)(y,u) = 1_{\{t<u\}}v(y,u-t)$
and $(\hH(s)v)(y,u)=\int_0^u e^{-st}v(y,u-t)\,dt$.
Hence,
\begin{align*}
|\hH(s)v|_\infty\le e^\delta |v|_{\delta,\infty}
\quad\text{and}\quad
|(\hH(s)v)(y,u)-(\hH(s)v)(y',u)|\le e^\delta |v|_{\delta,\eta}\,d(y,y')^\eta.
\end{align*}
The result follows.
\end{proof}

\begin{prop} \label{prop:BG}
There exists $\delta>\eps>0$, $C>0$ such that
\[
{\|\hB(s)\|}_{\cF_{\delta,\eta}(Y^\varphi)\to\cF_\eta(\tY)}\le C|s|
\quad\text{and}\quad {\|\hG(s)\|}_{\cF_{\delta,\eta}(Y^\varphi)\to\cF_\eta(\tY)}\le C|s|
\quad\text{for $|a| \le \eps$.}
\]
\end{prop}

\begin{proof}
Let $v\in L^1(Y^\varphi)$, $w\in L^\infty(\tY)$.
Using that $F_t(y,u)=(Fy,u+t-\varphi(y))$ for
$(y,u)\in\Delta_t$,
\begin{align*}
\int_\tY B_tv\,w\,d\tmu &
=\bar\varphi \int_{Y^\varphi}L_t(1_{\Delta_t}v)\,w\,d\mu^\varphi
=\bar\varphi \int_{Y^\varphi}1_{\Delta_t}v\,w\circ F_t\,d\mu^\varphi
\\
& = \int_Y\int_0^{\varphi(y)}1_{\{0\le u+t-\varphi(y)<1\}}v(y,u)w(Fy,u+t-\varphi)\,du\,d\mu
\\
& = \int_Y \int_{t-\varphi(y)}^t 1_{\{0\le u<1\}}v(y,u+\varphi(y)-t)w(Fy,u)\,du\,d\mu
\\
& = \int_\tY v_t\,w\circ \tF\,d\tmu
 = \int_\tY \tR v_t\,w\,d\tmu
\end{align*}
where $v_t(y,u)=1_{\{0<u+\varphi(y)-t<\varphi(y)\}}v(y,u+\varphi(y)-t)$.

Hence $B_tv=\tR v_t$ and 
it follows immediately that $G_tv=\tR(\omega(t) v)_t$.
But
\[
(\omega(t) v)_t(y,u)=1_{\{0<u+\varphi(y)-t<\varphi(y)\}}(\omega(t)v)(y,u+\varphi(y)-t)
=(1-u)v_t(y,u),
\]
so $(G_tv)(y,u)=(1-u)(B_tv)(y,u)$.

Next, 
$\hB(s)v=\tR \hV(s)$ where
\begin{align*}
\hV(s)(y,u)  = \int_0^\infty e^{-st}v_t(y,u)\,dt
& = \int_u^{u+\varphi(y)} e^{-st}v(y,u+\varphi(y)-t)\,dt
\\ & =\int_0^{\varphi(y)}e^{-s(\varphi(y)+u-t)}v(y,t)\,dt.
\end{align*}
It is immediate that
\begin{align} \label{eq:BG}
(\hG(s)v)(y,u)=(1-u)(\hB(s)v)(y,u).
\end{align}

Suppose that $\delta>\eps>0$ are fixed.
Let $v\in \cF_{\delta,\eta}(Y^\varphi)$.
Let $(y,u),\,(y',u)\in\tY_j$, $j\ge1$.  Then
\[
|\hV(s)(y,u)|  \le \int_0^{\varphi(y)}e^{-a \,(\varphi(y)+u-t)}|v|_{\delta,\infty}\,e^{\delta t}\,dt
 \ll  e^{\delta\varphi(y)}|v|_{\delta,\infty}
\]
and so $|1_{Y_j}\hV(s)|_\infty\ll e^{\delta|1_{Y_j}\varphi|_\infty}|v|_{\delta,\infty}$.

Next, suppose without loss that $\varphi(y')\le \varphi(y)$.
Then
\[
\hV(s)(y,u)-\hV(s)(y',u)  = J_1+J_2+J_3
\]
where
\begin{align*}
 J_1 & = \int_0^{\varphi(y)}(e^{-s(\varphi(y)+u-t)}-e^{-s(\varphi(y')+u-t)})v(y,t)\,dt, \\
 J_2 & = \int_0^{\varphi(y)}e^{-s(\varphi(y')+u-t)}(v(y,t)-v(y',t))\,dt,  \\
 J_3 & = \int_{\varphi(y')}^{\varphi(y)}e^{-s(\varphi(y')+u-t)}v(y',t)\,dt.
\end{align*}
For notational convenience we suppose that $a\in(-\eps,0)$ since the range $a\ge0$ is simpler. Using~\eqref{eq:phi},
\begin{align*}
|J_1| & \le \int_0^{\varphi(y)} e^{\eps(|1_{Y_j}\varphi|_\infty+1-t)}
|s||\varphi(y)-\varphi(y')||v|_{\delta,\infty}\,e^{\delta t}\,dt \\
 & \ll |s|\varphi(y) e^{\delta|1_{Y_j}\varphi|_\infty}\,d(Fy,Fy')^\eta |v|_{\delta,\infty}
  \ll |s|e^{2\delta|1_{Y_j}\varphi|_\infty}\,d(Fy,Fy')^\eta |v|_{\delta,\infty},
\\[.75ex]
|J_2| & \le \int_0^{\varphi(y)} e^{\eps(|1_{Y_j}\varphi|_\infty+1-t)}
|v|_{\delta,\eta}\,e^{\delta t} d(y,y')^\eta \,dt
  \ll e^{\delta|1_{Y_j}\varphi|_\infty}\,d(y,y')^\eta |v|_{\delta,\eta},
\\[.75ex]
|J_3| & \le   \int_{\varphi(y')}^{\varphi(y)} e^{\eps (|1_{Y_j}\varphi|_\infty+1-t)}
|v|_{\delta,\infty}\,e^{\delta t} \,dt
  \ll e^{2\delta|1_{Y_j}\varphi|_\infty}|v|_{\delta,\infty}\,
d(Fy,Fy')^\eta.
\end{align*}
Hence
\[
|\hV(s)(y,u)-\hV(s)(y,u)|\ll |s|e^{2\delta|1_{Y_j}\varphi|_\infty}\|v\|_{\delta,\eta} \,d(Fy,Fy')^\eta.
\]

Now, for
$(y,u)\in \tY$,
\[
(\hB(s)v)(y,u)=(\tR\hV(s))(y,u)=\sumj g(y_j)\hV(s)(y_j,u),
\]
where $y_j$ is the unique preimage of $y$ under $F|Y_j$.
It follows from the estimates for $\hV(s)$ together with~\eqref{eq:GM} 
that 
\[
\|\hB(s)v\|_\eta \ll |s|\sumj \mu(Y_j)e^{2\delta|1_{Y_j}\varphi|_\infty}
\|v\|_{\delta,\eta}.
\]
Shrinking $\delta$, the desired estimate for $\hB$ follows from~\eqref{eq:exp}.
Finally, the estimate for $\hG$ follows from~\eqref{eq:BG}.
\end{proof}

\begin{prop} \label{prop:zero}
$\int_\tY\hB(0)v\,d\tmu=\bar\varphi\int_{Y^\varphi}v\,d\mu^\varphi$
for $v\in L^1(Y^\varphi)$.
\end{prop}

\begin{proof}
By the definition of $\hB$,
\begin{align*}
\int_\tY & \hB(0)v\,d\tmu =
\int_\tY\int_0^\infty L_t(1_{\Delta_t}v)\,dt\,d\tmu=
\bar\varphi \int_0^\infty \int_{Y^\varphi} L_t(1_{\Delta_t}v)\,d\mu^\varphi\,dt
\\ & =
\bar\varphi\int_0^\infty\int_{Y^\varphi} 1_{\Delta_t}v\,d\mu^\varphi\,dt
 =\bar\varphi\int_{Y^\varphi}\int_0^\infty 1_{\{\varphi-u<t<\varphi-u+1\}}v\,dt\,d\mu^\varphi
 =\bar\varphi\int_{Y^\varphi}v\,d\mu^\varphi,
\end{align*}
as required.
\end{proof}

\begin{lemma} \label{lem:hatL}
Write $s=a+ib\in\C$.
There exists $\eps>0$, $\delta>0$, $m_2\ge0$, $C>0$ such that 
\begin{itemize}
\item[(a)] $s\mapsto \hL(s):\cF_{\delta,\eta}^0(Y^\varphi)\to\cF_{\delta,\eta}(Y^\varphi)$ is analytic on $\{|a|< \eps\}$;
\item[(b)] $s\mapsto \hL(s):\cF_{\delta,\eta}(Y^\varphi)\to\cF_{\delta,\eta}(Y^\varphi)$ is analytic on $\{|a|< \eps\}$ except for a simple pole at $s=0$;
\item[(c)]
$\|\hL(s)v\|_{\delta,\eta}
\le C|b|^{m_2}\|v\|_{\delta,\eta}$
for $|a|\le\eps$, $|b|\ge1$, $v\in \cF_{\delta,\eta}(Y^\varphi)$.
\end{itemize}
\end{lemma}

\begin{proof}  
Recall that
\[
\hL=\hA\hT\hB+\hA\hG+\hA\hH+\hE, \qquad \hT=\hU(I-\hR)^{-1}
\]
where $\hU$, $\hA$, $\hB$, $\hG$, $\hH$ and $\hE$ are analytic by
Propositions~\ref{prop:U},~\ref{prop:AEH} and~\ref{prop:BG}.
Hence part~(b) follows immediately from Lemma~\ref{lem:main}(b).
Also, part~(c) follows using Lemma~\ref{lem:main}(c).

By Proposition~\ref{prop:zero},
$\hB(0)(\cF_{\delta,\eta}^0(Y^\varphi))\subset \cF_\eta^0(\tY)$.
Hence the simple pole at $s=0$ for $(I-\hR)^{-1}\hB$ disappears on restriction
to $\cF_{\delta,\eta}^0(Y^\varphi)$
by Lemma~\ref{lem:main}(a).  This proves part~(a).
\end{proof}

\subsection{Moving the contour of integration}
\label{sec:contour}

\begin{prop} \label{prop:partial}
Let $m\ge1$. 
Let $v\in \cF_{\delta,\eta,m}(Y^{\varphi})$ with good support. Then
$\hL(s)v=\sum_{j=0}^{m-1} (-1)^js^{-(j+1)}\partial_t^j v+(-1)^ms^{-m}\hL(s)\partial_t^mv$
for $a >0$.
\end{prop}

\begin{proof}
Recall that $\supp v\subset \{(y,u)\in Y^\varphi:u\in[r,\varphi(y)-r]\}$ for some $r>0$.
For $h\in[0,r]$, we can define $(\Psi_hv)(y,u)=v(y,u-h)$ and then 
$(\Psi_hv)\circ F_h=v$.

Let $w\in L^\infty(Y^\varphi)$
and write $\rho_{v,w}(t)=\int_{Y^\varphi}v\,w_t\,d\mu^\varphi$ where
$w_t=w\circ F_t$.
Then for $h\in[0,r]$,
\[
\rho_{v,w}(t+h)=\int_{Y^\varphi}v\,w_t\circ F_h\,d\mu^\varphi
=\int_{Y^\varphi}(\Psi_hv)\circ F_h\,w_t\circ F_h\,d\mu^\varphi
=\int_{Y^\varphi}\Psi_hv\,w_t\,d\mu^\varphi.
\]
Hence
\(
h^{-1}(\rho_{v,w}(t+h)-\rho_{v,w}(t))=
\int_{Y^\varphi}h^{-1}(\Psi_hv-v)\,w_t\,d\mu^\varphi
\)
so 
\[
\rho'_{v,w}(t)=
-\int_{Y^\varphi}\partial_tv\,w_t\,d\mu^\varphi
=-\int_{Y^\varphi}\partial_tv\,w\circ F_t\,d\mu^\varphi
=-\rho_{\partial_tv,w}(t).
\]
Inductively,
$\rho^{(j)}_{v,w}(t)= (-1)^j\rho_{\partial_t^j v,w}(t)$.

Now
$\int_{Y^\varphi}\hL(s)v\,w\,d\mu^\varphi
 =\int_0^\infty e^{-st}\int_{Y^\varphi}L_tv\,w\,d\mu^\varphi\,dt
=\int_0^\infty e^{-st}\rho_{v,w}(t)\,dt$, so
repeatedly integrating by parts,
\begin{align*}
\int_{Y^\varphi}\hL(s)v\,w\,d\mu^\varphi
& =
 \sum_{j=0}^{m-1} s^{-(j+1)}\rho^{(j)}_{v,w}(0)+s^{-m}\int_0^\infty e^{-st}\rho^{(m)}_{v,w}(t)\,dt
\\ & = \sum_{j=0}^{m-1} (-1)^js^{-(j+1)}\rho_{\partial_t^jv,w}(0)+(-1)^ms^{-m}\int_0^\infty e^{-st}\rho_{\partial_t^m v,w}(t)\,dt
\\ & = \int_{Y^\varphi}\sum_{j=0}^{m-1} (-1)^js^{-(j+1)}\partial_t^jv\, w\,d\mu^\varphi
+(-1)^ms^{-m}\int_0^\infty e^{-st}\rho_{\partial_t^m v,w}(t)\,dt.
\end{align*}
Finally,
$\int_0^\infty e^{-st}\rho_{\partial_t^m v,w}(t)\,dt=
\int_{Y^\varphi}\hL(s)\partial_t^m v\, w\,d\mu^\varphi$ and the result follows since
$w\in L^\infty(Y^\varphi)$ is arbitrary.
\end{proof}

We can now estimate 
$\|L_tv\|_{\delta,\eta}$.

\begin{cor} \label{cor:partial}
Under the assumptions of Theorem~\ref{thm:normdecay},
there exists $\eps>0$, $m_3\ge1$, $C>0$ such that
\[
\|L_tv\|_{\delta,\eta}
\le Ce^{-\eps t}\|v\|_{\delta,\eta,m_3}
\quad\text{for all $t>0$}  
\]
for all $v\in \cF_{\delta,\eta,m_3}^0(Y^{\varphi})$ with good support.
\end{cor}

\begin{proof}
Let $m_3=m_2+2$.
By Lemma~\ref{lem:hatL}(a),
$\hL(s):\cF_{\delta,\eta,m_3}^0(Y^{\varphi})\to
\cF_{\delta,\eta}(Y^\varphi)$ is analytic for $|a|\le\eps$.
The alternative expression 
in Proposition~\ref{prop:partial} is also analytic on this region (the apparent singularity at $s=0$ is removable by the equality with the analytic function $\hL$).
Hence we can move the contour of integration to 
$s=-\eps+ib$ when computing the inverse Laplace transform, to obtain
\begin{align*}
L_tv & =\int_{-\infty}^\infty e^{st}
\Big(\sum_{j=0}^{m_3-1} (-1)^j s^{-(j+1)}\partial_t^j v+(-1)^{m_3} s^{-{m_3}}\hL(s)\partial_t^{m_3}v\Big)\,db
\\ & = e^{-\eps t}\sum_{j=0}^{m_3-1}(-1)^j\partial_t^j v\int_{-\infty}^\infty e^{ibt} s^{-(j+1)}\,db
+(-1)^{m_3} e^{-\eps t}\int_{-\infty}^\infty 
 e^{ibt} s^{-{m_3}}\hL(s)\partial_t^{m_3}v \,db.
\end{align*}
The final term is estimated using Lemma~\ref{lem:hatL}(b,c):
\[
\Big\|\int_{-\infty}^\infty 
 e^{ibt} s^{-{m_3}}\hL(s)\partial_t^{m_3}v \,db\Big\|_{\delta,\eta}
\ll \int_{-\infty}^\infty  (1+|b|)^{-(m_2+2)}(1+|b|)^{m_2}\|\partial_t^{m_3} v\|_{\delta,\eta}\,db
\ll \|v\|_{\delta,\eta,{m_3}}.
\]
Clearly, the integrals $\int_{-\infty}^\infty e^{ibt}s^{-(j+1)}\,db$ converge absolutely for $j\ge1$, while the integral for $j=0$ converges as an improper Riemann integral.
Hence altogether we obtain that $\|L_tv\|_{\delta,\eta}
\ll e^{-\eps t}\|v\|_{\delta,\eta,m_3}$.
\end{proof}

For the proof of Theorem~\ref{thm:normdecay},
it remains to estimate $\|\partial_uL_tv\|_{\delta,\eta}$.
Recall that the transfer operator $R_0$ for $F$ has weight function $g$. We have the pointwise formula
$(R_0^kv)(y)=\sum_{F^ky'=y}g_k(y')v(y')$ where 
$g_k=g\,\dots\,g\circ F^{k-1}$.
Let $\varphi_k=\sum_{j=0}^{k-1}\varphi\circ F^j$.

\begin{prop} \label{prop:Lt}
Let $v\in L^1(Y^\varphi)$. Then 
for all $t>0$, $(y,u)\in Y^\varphi$,
\[
(L_tv)(y,u)=\sum_{k=0}^{[t/2]} \sum_{F^ky'=y}g_k(y')1_{\{0\le u-t+\varphi_k(y')<\varphi(y')\}}v(y',u-t+\varphi_k(y')).
\]
\end{prop}

\begin{proof}
The lap number $N_t(y,u)\in[0,t/2]\cap\N$ is the unique integer $k\ge0$ such that
$u+t-\varphi_k(y)\in[0,\varphi(F^ky))$.
In particular, $F_t(y,u)=(F^{N_t(y,u)}y,u+t-\varphi_{N_t(y,u)}(y))$.
For $w\in L^\infty(Y^\varphi)$,
\begin{align*}
\int_{Y^\varphi} & L_t(1_{\{N_t=k\}} v)\,w\,d\mu^\varphi
  = 
\int_{Y^\varphi}1_{\{N_t=k\}} v\,w\circ F_t\,d\mu^\varphi
 \\ & = \bar\varphi^{-1}
\int_Y \int_0^{\varphi(y)}1_{\{0\le u+t-\varphi_k(y)<\varphi(F^ky)\}} v(y,u)\,w(F^ky,u+t-\varphi_k(y))\,du\,d\mu
 \\ & = \bar\varphi^{-1}
\int_Y \int_0^{\varphi(F^ky)}1_{\{0\le u-t+\varphi_k(y)<\varphi(y)\}} v(y,u-t+\varphi_k(y))\,w(F^ky,u)\,du\,d\mu.
\end{align*}
Writing $v_{t,k}^u(y)=1_{\{0\le u-t+\varphi_k(y)<\varphi(y)\}}v(y,u-t+\varphi_k(y))$ and $w^u(y)=w(y,u)$,
\begin{align*}
\int_{Y^\varphi} & L_t(1_{\{N_t=k\}} v)\,w\,d\mu^\varphi
  = \bar\varphi^{-1} \int_0^\infty
\int_Y 1_{\{u<\varphi\circ F^k\}}v_{t,k}^u\,w^u\circ F^k\,d\mu\,du
 \\ & = \bar\varphi^{-1} \int_0^\infty
\int_Y 1_{\{u<\varphi\}}R_0^k v_{t,k}^u\,w^u\,d\mu\,du
 = \int_{Y^\varphi}(R_0^k v_{t,k}^u)(y)\,w(y,u)\,d\mu^\varphi.
\end{align*}
Hence,
\begin{align*}
(L_t v)(y,u) = 
\sum_{k=0}^{[t/2]} (L_t(1_{\{N_t=k\}} v)(y,u)
= \sum_{k=0}^{[t/2]} (R_0^k v_{t,k}^u)(y).
\end{align*}
The result follows from the pointwise formula for $R_0^k$.
\end{proof}

\begin{pfof}{Theorem~\ref{thm:normdecay}}
Let $m=m_3+1$.
By Corollary~\ref{cor:partial}, $\|L_tv\|_{\delta,\eta}\ll e^{-\eps t}\|v\|_{\delta,\eta,m}$.

Recall that $\partial_u$ denotes the ordinary derivative with respect to $u$ at $0<u<\varphi(y)$ and denotes the appropriate one-sided derivative at $u=0$ and $u=\varphi(y)$.
Since $v$ has good support,
the indicator functions in the right-hand side of the formula in Proposition~\ref{prop:Lt} are constant on the support of $v$.
It follows that
$\partial_u L_tv=L_t(\partial_uv)$.  By Corollary~\ref{cor:partial},
\[
\|\partial_u L_tv\|_{\delta,\eta}
=\|L_t(\partial_u v)\|_{\delta,\eta}
\ll e^{-\eps t}\|\partial_uv\|_{\delta,\eta,m_3}
\le e^{-\eps t}\|v\|_{\delta,\eta,m}.
\]
Hence,
$\|L_tv\|_{\delta,\eta,1}
\ll e^{-\eps t}\|v\|_{\delta,\eta,m}$ as required.
\end{pfof}



\begin{thebibliography}{10}

\bibitem{Aaronson}
J.~Aaronson. \emph{{An Introduction to Infinite Ergodic Theory}}. Math. Surveys
  and Monographs \textbf{50}, Amer. Math. Soc., 1997.

\bibitem{AD01}
J.~Aaronson and M.~Denker. {Local limit theorems for partial sums of stationary
  sequences generated by Gibbs-Markov maps}. \emph{Stoch. Dyn.} \textbf{1}
  (2001) 193--237.

\bibitem{ADU93}
J.~Aaronson, M.~Denker and M.~Urba{\'n}ski. Ergodic theory for {M}arkov fibred
  systems and parabolic rational maps. \emph{Trans. Amer. Math. Soc.}
  \textbf{337} (1993) 495--548.

\bibitem{AntoniouM19}
M.~Antoniou and I.~Melbourne. Rate of convergence in the weak invariance
  principle for deterministic systems. \emph{Comm. Math. Phys.} (2019)
  1147--1165.

\bibitem{AraujoM16}
V.~Ara{\'u}jo and I.~Melbourne. Exponential decay of correlations for
  nonuniformly hyperbolic flows with a {$C^{1+\alpha}$} stable foliation,
  including the classical {L}orenz attractor. \emph{Ann. Henri Poincar\'e}
  \textbf{17} (2016) 2975--3004.

\bibitem{AGY06}
A.~Avila, S.~Gou{\"e}zel and J.~Yoccoz. Exponential mixing for the
  {T}eichm\"uller flow. \emph{Publ. Math. Inst. Hautes \'Etudes Sci.}
  \textbf{104} (2006) 143--211.


\bibitem{BaladiVallee05}
V.~Baladi and B.~Vall{\'e}e. Exponential decay of correlations for surface
  semi-flows without finite {M}arkov partitions. \emph{Proc. Amer. Math. Soc.}
  \textbf{133} (2005) 865--874.

\bibitem{Bowen75}
R.~Bowen. \emph{{Equilibrium States and the Ergodic Theory of Anosov
  Diffeomorphisms}}. Lecture Notes in Math. \textbf{470}, Springer, Berlin,
  1975.

\bibitem{BMT19}
H.~Bruin, I.~Melbourne and D.~Terhesiu. {Rates of mixing for nonMarkov
  infinite measure semiflows}. \emph{Trans. Amer. Math. Soc.} \textbf{371}
  (2019) 7343--7386.


\bibitem{Dolgopyat98a}
D.~Dolgopyat. {On the decay of correlations in Anosov flows}. \emph{Ann. of
  Math.} \textbf{147} (1998) 357--390.

\bibitem{GoraBoyarsky89}
P.~G{\'o}ra and A.~Boyarsky. Absolutely continuous invariant measures for
  piecewise expanding {$C^2$} transformation in {${\bf R}^N$}. \emph{Israel J.
  Math.} \textbf{67} (1989) 272--286.

\bibitem{Gouezel05}
S.~Gou{\"e}zel. Berry-{E}sseen theorem and local limit theorem for non
  uniformly expanding maps. \emph{Ann. Inst. H. Poincar\'e Probab. Statist.}
  \textbf{41} (2005) 997--1024.

\bibitem{HofbauerKeller82}
F.~Hofbauer and G.~Keller. {Ergodic properties of invariant measures for
  piecewise monotonic transformations}. \emph{Math. Z.} \textbf{180} (1982)
  119--140.

\bibitem{Keller85}
G.~Keller. Generalized bounded variation and applications to piecewise
  monotonic transformations. \emph{Z. Wahrsch. Verw. Gebiete} \textbf{69}
  (1985) 461--478.

\bibitem{Liverani04}
C.~Liverani. {On contact Anosov flows}. \emph{Ann. of Math.} \textbf{159}
  (2004) 1275--1312.

\bibitem{Liverani13}
C.~Liverani. Multidimensional expanding maps with singularities: a pedestrian
  approach. \emph{Ergodic Theory Dynam. Systems} \textbf{33} (2013) 168--182.

\bibitem{MPTapp}
I.~Melbourne, N.~Paviato and D.~Terhesiu. {Nonexistence of spectral gaps in
  H\"older spaces for continuous time dynamical systems}. \emph{Israel J.
  Math.}, to appear.

\bibitem{MT17}
I.~Melbourne and D.~Terhesiu. {Operator renewal theory for continuous time
  dynamical systems with finite and infinite measure}. \emph{{Monatsh. Math.}}
  \textbf{182} (2017) 377--431.

\bibitem{Pollicott85}
M.~Pollicott. {On the rate of mixing of Axiom A flows}. \emph{Invent. Math.}
  \textbf{81} (1985) 413--426.

\bibitem{Ruelle78}
D.~Ruelle. \emph{{Thermodynamic Formalism}}. Encyclopedia of Math. and its
  Applications \textbf{5}, Addison Wesley, Massachusetts, 1978.

\bibitem{Rychlik83}
M.~Rychlik. Bounded variation and invariant measures. \emph{Studia Math.}
  \textbf{76} (1983) 69--80.

\bibitem{Saussol00}
B.~Saussol. Absolutely continuous invariant measures for multidimensional
  expanding maps. \emph{Israel J. Math.} \textbf{116} (2000) 223--248.

\bibitem{Sinai72}
Y.~G. Sina{\u\i}. {Gibbs measures in ergodic theory}. \emph{Russ. Math. Surv.}
  \textbf{27} (1972) 21--70.

\bibitem{Tsujii08}
M.~Tsujii. Decay of correlations in suspension semi-flows of angle-multiplying
  maps. \emph{Ergodic Theory Dynam. Systems} \textbf{28} (2008) 291--317.

\bibitem{Tsujii10}
M.~Tsujii. Quasi-compactness of transfer operators for contact {A}nosov flows.
  \emph{Nonlinearity} \textbf{23} (2010) 1495--1545.

\bibitem{Young98}
L.-S. Young. Statistical properties of dynamical systems with some
  hyperbolicity. \emph{Ann. of Math.} \textbf{147} (1998) 585--650.

\bibitem{Young99}
L.-S. Young. Recurrence times and rates of mixing. \emph{Israel J. Math.}
  \textbf{110} (1999) 153--188.

\end{thebibliography}
\end{document}